\documentclass[12pt]{article}

\usepackage{amsmath,enumerate,amsfonts,amssymb,color,graphicx,amsthm}

\def\RR{{\mathbb R}}
\def\HH{{\mathbb H}}

\def\Snn{{\mathcal{S}^{n \times n}}}

\newcounter{marnote}

\begin{document}

\newtheorem{thm}{Theorem}[section]
\newtheorem{Def}[thm]{Definition}
\newtheorem{lem}[thm]{Lemma}
\newtheorem{rem}[thm]{Remark}
\newtheorem{question}[thm]{Question}
\newtheorem{prop}[thm]{Proposition}
\newtheorem{cor}[thm]{Corollary}
\newtheorem{example}[thm]{Example}
\newtheorem*{strong}{Strong comparison principle}
\newtheorem{theoremAlph}{Theorem}\renewcommand\thetheoremAlph{\Alph{theoremAlph}}

\title{Comparison principles for some fully nonlinear subelliptic equations on the Heisenberg group}

\author{YanYan Li \footnote{Department of Mathematics, Rutgers University, 110 Frelinghuysen Rd, Piscataway, NJ 08854, USA. Email: yyli@math.rutgers.edu.}~\quad Bo Wang \footnote{
School of Mathematics and Statistics, Beijing Institute of Technology, Beijing 100081, China. Corresponding author. Email: wangbo89630@bit.edu.cn.}}

\date{}

\maketitle

\begin{abstract}
In this paper, we prove a form of the strong comparison principle for a class of fully nonlinear subelliptic operators of the form $\nabla^2_{H,s}\psi + L(\cdot,\psi,\nabla_{H}\psi)$ on the Heisenberg group, which include the CR invariant operators.



\end{abstract}

\setcounter{section}{0}

\section{Introduction}

In this paper, we establish a form of the comparison principle for a class of subelliptic equations on the Heisenberg group. 

Let $\Omega$ be an open connected subset of $\RR^{n}$ ($n\geq1$), the $n$-dimensional Euclidean space. Assume that  $u$, $v\in C^{2}(\Omega)$ satisfy
\begin{equation}\label{2011}
u\geq v,\quad\mbox{in }\Omega.
\end{equation}
The standard form of the strong comparison principle for nonlinear second order elliptic operators $F(x,u,\nabla u,\nabla^{2}u)$ is the following. Here $F(x,s,p,M)$ is of class $C^{1}$, $x\in\Omega$, $s\in\RR^{1}$, $p\in\RR^{n}$, $M\in\mathcal{S}^{n\times n}$, the set of all $n\times n$ real symmetric matrices, and is elliptic, i.e., 
\begin{equation*}
\frac{\partial F}{\partial M_{ij}}>0.
\end{equation*}
\begin{strong}
Let $u$, $v\in C^{2}(\Omega)$ satisfy (\ref{2011}) and 
\begin{equation*}
F(x,u,\nabla u,\nabla^{2}u)\leq F(x,v,\nabla v,\nabla^{2}v),\quad\mbox{in }\Omega.
\end{equation*}
Then we have 
\begin{equation*}
\mbox{either }u>v\mbox{ in }\Omega,\quad\mbox{or }u\equiv v\mbox{ in }\Omega.
\end{equation*}
\end{strong}

In \cite{CLN2009} and \cite{CLN2013}, L. Caffarelli, L. Nirenberg and the first named author obtained some forms of the comparison principle for singular solutions of the nonlinear elliptic operators of the form $F(x,u,\nabla u,\nabla^{2}u)$.

In recent years, comparison principles for degenerate elliptic equations have been widely studied; see \cite{AmendolaGaliseVitolo13-DIE}-\cite{BirindelliDemengel07-CPAA}, \cite{DolcettaVitolo07-MC}-\cite{Trudinger88-RMI} and the references therein. One type of those equations, which appeared in \cite{CLN2009}-\cite{CNS1986}, \cite{HarveyLawsonSurvey2013,HarveyLawson2016}, \cite{Li07-ARMA,Li09-CPAM}, \cite{LN,LW}, involve a symmetric matrix function
\begin{equation}
G[u]:=\nabla^{2}u+L(x,u,\nabla u),\label{1988}
\end{equation}
where $L\in C^{0,1}(\Omega\times\RR\times\RR^{n})$, is in $\mathcal{S}^{n\times n}$.

One such matrix operator is the conformal Hessian matrix operator ( see e.g. \cite{LiLi2003}, \cite{V} and the references therein), i.e.,
\begin{equation*}
H[u]=\nabla^{2}u+\nabla u\otimes \nabla u-\frac{1}{2}|\nabla u|^{2}I_{n},
\end{equation*}
where $I_{n}$ denotes the $n\times n$ identity matrix, and, for $p$, $q\in\RR^{n}$, $p\otimes q$ denotes the $n\times n$ matrix with entries $(p\otimes q)_{ij} =p_{i}q_{j}$ for $i$, $j=1$, $\cdots$, $n$.

Let $U$ be an open subset of $\mathcal{S}^{n\times n}$ satisfying
\begin{equation*}
0\in\partial U,\quad U+\mathcal{P}\subset U,\quad O^{t}UO\subset U,~\forall~O\in O(n)\quad tU\subset U,~\forall~t>0,
\end{equation*}
where $\mathcal{P}$ denotes the set of all $n\times n$ non-negative real matrices and $O(n)$ denotes the set of all $n\times n$ real orthogonal matrices.

In \cite{LN}, it was shown that, under the assumption
\begin{equation}\label{1990}
\mbox{diag}\{1,0,\cdots,0\}\in\partial U,
\end{equation}
the strong comparison principle and Hopf Lemma fail for a class of equations of the form
\begin{equation*}
G[u]\in\partial U.
\end{equation*}
Conversely, if (\ref{1990}) does not hold, then the strong comparison principle and Hopf Lemma holds.

Although the strong comparison principle fails under assumption (\ref{1990}), the first named author proved that a weak form of strong comparison principle holds for the conformal Hessian operator $H[u]$ and locally Lipschitz continuous solutions in \cite{Li09-CPAM}. This comparison principle has played an important role in deriving local gradient estimates and symmetry properties for solutions to (both degenerate and non-degenerate elliptic) equations arising from studies in conformal geometry; see \cite{Li09-CPAM} and the references therein. Later on, in \cite{LiNgWang-arxiv}, this type of comparison principle was generalized to semi-continuous solutions and a larger class of operators $G[u]$ with $L$ of the form
\begin{equation*}
L(x,s,p):=\alpha(x,s)p\otimes p-\beta(x,s)|p|^{2}I_{n},
\end{equation*}
where $\alpha$, $\beta:\Omega\times\RR$ satisfy 
\begin{equation*}
L(x, s, p)\mbox{ is non-decreasing in }s,
\end{equation*}
and 
\begin{equation*}
\begin{array}{ll}
& \text{either $|\beta(x, s)|> \beta_0 > 0$ for some constant $\beta_0$},\\
& \text{or both $\alpha$ and $\beta$ are constant}.
\end{array}
\end{equation*}

By taking $\alpha\equiv1$ and $\beta\equiv1$, operator $G[u]$ becomes $H[u]$. The comparison principle was applied in \cite{LiNgWang-arxiv} to obtain the local Lipschitz regularity of viscosity solutions of fully nonlinear degenerate conformally invariant equations. Since this type of comparison principle also includes the equations arising from fully nonlinear Yamabe problem of ''negative type'', as  another application, it was recently applied to a fully nonlinear version of the Loewner-Nirenberg problem in \cite{GLN2018}. 

In \cite{LiMonticelli-JDE}, the first named author and D. D. Monticelli considered the comparison principle for solutions of second order fully nonlinear CR invariant equations. Let $\Omega$ be an open bounded subset of $\mathbb{H}^{n}$, the $n$-dimensional Heisenberg group. For any $C^2$ positive function $u$ in $\Omega$, it was proved in \cite{LiMonticelli-JDE} that a second order fully nonlinear operator is CR invariant if and only it has the form 
\begin{align*}
A^{u}&:=-\frac{2}{Q-2}u^{-\frac{Q+2}{Q-2}}\nabla^{2}_{H,s}u+\frac{2Q}{(Q-2)^{2}}u^{-\frac{2Q}{Q-2}}\nabla_{H}u\otimes \nabla_{H}u\\
&\quad\quad\quad\quad-\frac{4}{(Q-2)^{2}}u^{-\frac{2Q}{Q-2}}J\nabla_{H}u\otimes J\nabla_{H}u-\frac{2}{(Q-2)^{2}}u^{-\frac{2Q}{Q-2}}|\nabla_{H}u|^{2}I_{2n},
\end{align*}
where $Q=2n+2$ denote the homogenous dimension of $\mathbb{H}^{n}$, $\nabla^{2}_{H,s}u$ and $\nabla_{H}u$ denote the symmetrized Heisenberg Hessian matrix and Heisenberg gradient of $u$, respectively (see Subsection 2.1), and \[
J:=\left(
\begin{array}{cc}
  0_{n}&  I_{n}  \\
 -I_{n}& 0_{n} 
\end{array}
\right).
\]
For geometric aspect related to the CR invariant operators, we refer to \cite{BCM12} and \cite{MM}.

It was proved in \cite{LiMonticelli-JDE} the comparison principle for the equations of the form 
\begin{equation*}
A^{u}\in\partial\Sigma,\quad\mbox{in }\Omega,
\end{equation*}
where $\Sigma$ is a non-empty open subset of $\mathcal{S}^{n\times n}$, satisfying a degenerate ellipticity condition:
\begin{equation}
\text{ if }A\in \Sigma, B\in\mathcal{S}^{2n\times 2n}\mbox{ and }B>0, \text{ then } A+B\in \Sigma.
	\label{Eq:UCondPos}
\end{equation}
(Note that \eqref{Eq:UCondPos} implies that $\partial \Sigma$ is Lipschitz.) and a homogeneity condition:
\begin{equation}
A \in \Sigma \text{ and } c > 0 \Rightarrow cA \in \Sigma.
	\label{Eq:UCone}
\end{equation}

\begin{theoremAlph}(\cite{LiMonticelli-JDE})\label{LM}
Let $\Sigma$ satisfy (\ref{Eq:UCondPos}) and (\ref{Eq:UCone}). Assume that $u$, $v\in C^{2}(\Omega)\cap C^{0}(\overline{\Omega})$, $u$, $v>0$ in $\Omega$ and satisfy 
\begin{equation*}
A^{u}\in\overline{\Sigma},\quad A^{v}\in \mathcal{S}^{2n\times 2n}\setminus \Sigma,\quad\mbox{in }\Omega.
\end{equation*}
Then 

\begin{enumerate}[(a)]
\item If $u \geq v$ on $\partial\Omega$, then $u\geq v$ in $\Omega$.

\item If $u > v$ on $\partial\Omega$, then $u> v$ in $\Omega$.
\end{enumerate}

\end{theoremAlph}

The main goal of our paper is to generalize Theorem \ref{LM} to semi-continuous viscosity solutions and to more general fully nonlinear subelliptic operators.

For any $C^2$ function $\psi$ in $\Omega$, we consider a symmetric matrix function
\begin{equation}
F[\psi] := \nabla^2_{H,s} \psi +L(\cdot,\psi,\nabla_{H}\psi),
	\label{Eq:FConstab}
\end{equation}
where $L\in C^{0,1}_{loc}(\Omega\times\RR\times\RR^{2n})$ is in $\mathcal{S}^{2n\times 2n}$, and is of the form 
\begin{equation}
L(\cdot,s,p) = \alpha(\cdot, s)\,p\otimes p -\gamma(\cdot, s)\, J  p \otimes J p - \beta(\cdot, s) |p|^2\,I_{2n}
	\label{Eq:FIntroDef}
\end{equation}
where $\alpha$, $\beta$, $\gamma\in C^{0,1}_{loc}(\Omega\times\RR)$. If $\alpha$, $\gamma\equiv1$ and $\beta\equiv\frac{1}{2}$, operator $F[\psi]$ becomes the operator
\begin{equation*}
A[\psi]=\nabla^2_{H,s} \psi+\nabla_{H}\psi \otimes \nabla_{H}\psi-J  \nabla_{H}\psi \otimes J  \nabla_{H}\psi - \frac{1}{2} |\nabla_{H} \psi|^{2}I_{2n}\end{equation*}
By letting $u = e^{-\frac{Q-2}{2}\psi}$, it is easy to see that $A^u=e^{2\psi} A[\psi]$.

We consider the equation 
\[
F[\psi] \in \partial \Sigma.
\]

For any set $S \subset\mathbb{H}^{n}$, we use $\mbox{USC}(S)$ to denote the set of functions $\psi:S\rightarrow\mathbb{R}\cup\{-\infty\}$, $\psi \not\equiv -\infty$ in $S$, satisfying 
\begin{equation*}
\limsup\limits_{\xi\rightarrow\bar{\xi}}\psi(\xi)\leq \psi(\bar{\xi}),\quad \forall \bar{\xi}\in S.
\end{equation*}
Similarly, we use $\mbox{LSC}(S)$ to denote the set of functions $\psi: S\rightarrow\mathbb{R}\cup\{+\infty\}$, $\psi \not\equiv +\infty$  in $S$, satisfying 
\begin{equation*}
\liminf\limits_{\xi\rightarrow\bar{\xi}}\psi(\xi)\geq \psi(\bar{\xi}),\quad \forall \bar{\xi}\in S.
\end{equation*}

We now give the definition of viscosity subsolutions, supersolutions and solutions to the subelliptic equation $F[\psi] \in \partial \Sigma$.

\begin{Def}\label{Def:ViscositySolution}
Let $\Omega\subset\mathbb{H}^{n}$, $n \geq 1$, be an open set, and $\Sigma$ be a non-empty open 
subset of $\mathcal{S}^{2n\times 2n}$ satisfying \eqref{Eq:UCondPos}. For a function  $\psi$ in $USC(\Omega)$ ($LSC(\Omega)$), we say that 
\begin{equation*}
F[\psi] \in \overline{\Sigma}\quad \left(F[\psi]\in\Snn\setminus \Sigma\right) \quad\mbox{in }\Omega \quad\mbox{in the viscosity sense}
\end{equation*}
if for any $\xi_{0}\in\Omega$, $\varphi\in C^{2}(\Omega)$, $(\psi-\varphi)(\xi_{0})=0$ and 
\begin{equation*}
\psi-\varphi\leq0\quad(\psi-\varphi\geq0),\quad\mbox{near }\xi_{0},
\end{equation*}
there holds
\begin{equation*}
F[\varphi](\xi_{0})\in \overline{\Sigma}\quad \left(F[\varphi](\xi_{0})\in\mathcal{S}^{2n\times 2n}\setminus \Sigma\right).
\end{equation*}

We say that a function $\psi \in C^0(\Omega)$ satisfies 
\begin{equation}
F[\psi]\in \partial \Sigma \text{ in the viscosity sense}
	\label{Eq:FpsiEq}
\end{equation}
in $\Omega$ if $F[\psi]$ belongs to both $\overline{\Sigma}$ and $\mathcal{S}^{2n\times 2n}\setminus \Sigma$ in $\Omega$ in the viscosity sense.
 
 When $F[\psi] \in \overline{\Sigma}\quad \left(F[\psi]\in\mathcal{S}^{2n\times 2n}\setminus \Sigma\right)$ in $\Omega$ in the viscosity sense, we also say interchangeably that $\psi$ is a viscosity subsolution (supersolution) to \eqref{Eq:FpsiEq} in $\Omega$.
\end{Def}

In the sequel, we say that \emph{the principle of propagation of touching points} holds for $(F,\Sigma)$ if for any supersolution $w \in LSC(\bar\Omega)$ and subsolution $v\in USC(\bar\Omega)$ of \eqref{Eq:FpsiEq} in $\Omega$ one has
\[
w\geq v \mbox{ in } \Omega \text{ and }  w > v \text{ on } \partial\Omega \qquad\Rightarrow \qquad w > v \text{ in } \Omega.
\]
(In other words, if $w \geq v$ in $\Omega$ then every non-empty connected component of the set $\{\xi \in \bar\Omega: w(\xi) = v(\xi)\}$ contains a point of $\partial\Omega$.) This principle can be viewed as a weak version of the strong comparison principle.

We say that \emph{the comparison principle} holds for $(F,\Sigma)$ if for any supersolution $w \in LSC(\bar\Omega)$ and subsolution $v\in USC(\bar\Omega)$ of \eqref{Eq:FpsiEq} in $\Omega$ one has
\[
w \geq v \text{ on } \partial\Omega \qquad \Rightarrow \qquad w \geq v \text{ in }  \Omega.
\]
It should be noted that, for general degenerate elliptic equations, $w \geq v$ in $\Omega$ does not imply the dichotomy that $w > v$ or $w \equiv v$ in $\Omega$. (This is in contrast with the uniformly elliptic case.)

\begin{rem}\label{rem:SCP=>CP}
If $L(\xi,s,p)$ is independent of $s$, then the principle of propagation of touching points is equivalent to the comparison principle. 
\end{rem}

We prove that the principle of propagation of touching points holds under the following structural conditions:

\begin{equation}
\text{$L(\xi,s,p)$ is non-decreasing in $s$.}
	\label{Eq:LMonotone}
\end{equation}
and 
\begin{equation}
\begin{array}{ll}
& \text{$\beta(\xi, s)> \beta_0 > 0$ for some constant $\beta_0$},\gamma(\xi,s)\geq0,\\
\text{ or } &  \text{$\beta(\xi, s)<- \beta_0 < 0$ for some constant $\beta_0$},\gamma(\xi,s)\leq0,\\
\text{ or } & \text{ $\alpha$ and $\beta$ are constant, $\gamma\equiv0$}.
\end{array}
	\label{Eq:betaStruct}
\end{equation}
Note that the conditions (\ref{Eq:LMonotone}) and (\ref{Eq:betaStruct}) are consistent with $A[\psi]$ defined as above.

\begin{thm}[Principle of propagation of touching points]\label{thm:CPQuad}
Let $F$ be of the form \eqref{Eq:FIntroDef} with $\alpha, \beta, \gamma\in C^{0,1}_{loc}(\bar\Omega \times \RR)$ satisfying \eqref{Eq:LMonotone} and \eqref{Eq:betaStruct}. Let $\Omega\subset\mathbb{H}^{n}$ ($n \geq 1$) be a bounded open set, and $\Sigma$ be a non-empty open subset of $\mathcal{S}^{2n \times 2n}$ satisfying \eqref{Eq:UCondPos} and \eqref{Eq:UCone}. Assume that $w \in LSC(\bar\Omega)$ and $v\in USC(\bar\Omega)$ are respectively a supersolution and a subsolution of \eqref{Eq:FpsiEq} in $\Omega$.

\begin{enumerate}[(a)]
\item If $w \geq v$ in $\Omega$ and $w > v$ on $\partial\Omega$, then $w > v$ in $\Omega$.

\item In case $\alpha$, $\beta$ and $\gamma$ are constant, if $w \geq v$ on $\partial\Omega$, then $w \geq v$ in $\Omega$.
\end{enumerate}
\end{thm}

As a corollary of Theorem \ref{thm:CPQuad}, we have 

\begin{thm}[Uniqueness for the Dirichlet Problem]
Let $\Omega\subset\mathbb{H}^{n}$ ($n \geq 1$) be a bounded open set, and $\Sigma$ be a non-empty open subset of $\mathcal{S}^{2n\times 2n}$ satisfying \eqref{Eq:UCondPos} and \eqref{Eq:UCone}. Assume that $F$ is of the form \eqref{Eq:FConstab} with constants $\alpha$, $\beta$, $\gamma$ satisfying (\ref{Eq:betaStruct}). Then, for any $\varphi \in C^0(\partial\Omega)$, there exists at most one solution $\psi \in C^{0}(\bar\Omega)$ of \eqref{Eq:FpsiEq} satisfying $\psi = \varphi$ on $\partial\Omega$.
\label{thm:Uniq}
\end{thm} 

We also prove the following existence theorem using Perron's method (see \cite{Ishii89-CPAM}). 

\begin{thm}[Existence by sub- and supersolution method]\label{thm:Perron}
Let $\Omega$ and $(F,\Sigma)$ be as in Theorem \ref{thm:Uniq}. Let
$w\in LSC(\overline \Omega)$ and $v\in USC(\overline\Omega )$ be respectively supersolution and subsolution of \eqref{Eq:FpsiEq} in $\Omega$ such that $w\geq v$ in $\Omega$ and $w=v$ on $\partial\Omega$. Then there exists a viscosity solution $u\in C^0(\overline \Omega)$ of \eqref{Eq:FpsiEq} in $\Omega$ satisfying
\begin{align*}
v\le u\le w &\qquad\mbox{in}\ \Omega,\\
u=w=v &\qquad \mbox{on}\ \partial\Omega.
\end{align*}
\end{thm}

%+++++++++++++++++%
\section{Preliminaries}\label{Sec:Prelim}

\subsection{}

In this subsection, we briefly review some basic notations on the Heisenberg group.

The Heisenberg group $\mathbb{H}^{n}$ ($n\geq1$) is the set $\mathbb{R}^{n}\times\mathbb{R}^{n}\times\mathbb{R}$ endowed with the group action $\circ$ defined by 
\begin{equation*}
\xi\circ\hat{\xi}:=\left(x+\hat{x},y+\hat{y},t+\hat{t}+2\sum\limits_{i=1}^{n}(y_{i}\hat{x}_{i}-x_{i}\hat{y}_{i})\right)
\end{equation*}
for any $\xi=(x,y,t)$, $\hat{\xi}=(\hat{x},\hat{y},\hat{t})$ in $\mathbb{H}^{n}$, with $x=(x_{1},\cdots,x_{n})$, $\hat{x}=(\hat{x}_{1},\cdots,\hat{x}_{n})$, $y=(y_{1},\cdots,y_{n})$ and $\hat{y}=(\hat{y}_{1},\cdots,\hat{y}_{n})$ denoting elements of $\mathbb{R}^{n}$. We will also use the notation $\xi=(z,t)$ with $z=x+iy$, $z\in\mathbb{C}^{n}\simeq\mathbb{R}^{n}\times\mathbb{R}^{n}$. We consider the norm on $\mathbb{H}^{n}$ defined by 
\begin{equation*}
|\xi|_{H}:=\left[\left(\sum\limits_{i=1}^{n}(x_{i}^{2}+y_{i}^{2})\right)^{2}+t^{2}\right]^{\frac{1}{4}}=\left(|z|^{4}+t^{2}\right)^{\frac{1}{4}}.
\end{equation*}
The corresponding distance on $\mathbb{H}^{n}$ is defined accordingly by setting 
\begin{equation*}
d_{H}(\xi,\hat{\xi}):=|\hat{\xi}^{-1}\circ\xi|_{H},
\end{equation*}
where $\hat{\xi}^{-1}$ is the inverse of $\hat{\xi}$ with respect to $\circ$, i.e. $\hat{\xi}^{-1}=-\hat{\xi}$. For every $\xi\in\mathbb{H}^{n}$ and $R>0$, we will use the notation 
\begin{equation*}
D_{R}(\xi):=\{\eta\in\mathbb{H}^{n}|d_{H}(\xi,\eta)<R\}.
\end{equation*}

The vector fields 
\begin{align*}
X_{j}&:=\frac{\partial}{\partial x_{j}}+2y_{j}\frac{\partial}{\partial t},\quad j=1,\cdots,n,\\
Y_{j}&:=\frac{\partial}{\partial y_{j}}-2x_{j}\frac{\partial}{\partial t},\quad j=1,\cdots,n,\\
T&:=\frac{\partial}{\partial t}.
\end{align*}
form a base of the Lie algebra of vector fields on the Heisenberg group which are left invariant with respect to the group action $\circ$. For a regular function $u$ defined on a domain in $\mathbb{H}^{n}$, let $\nabla_{H}u$ denote the Heisenberg gradient, or horizontal gradient, of $u$, i.e.
\begin{equation*}
\nabla_{H}u:=(X_{1}u,\cdots,X_{n}u,Y_{1}u,\cdots,Y_{n}u),
\end{equation*}
while let $\nabla_{H}^{2}u$ to denote the Heisenberg Hessian matrix of $u$, i.e.
 \begin{eqnarray*}
\nabla_{H}^{2}u:=\left(\begin{array}{cccccc}
X_{1}X_{1}u&\cdots&X_{n}X_{1}u&Y_{1}X_{1}u &\cdots&Y_{n}X_{1}u\\
\cdots&\cdots&\cdots&\cdots&\cdots&\cdots\\
X_{1}X_{n}u&\cdots&X_{n}X_{n}u&Y_{1}X_{n}u &\cdots&Y_{n}X_{n}u\\
X_{1}Y_{1}u&\cdots&X_{n}Y_{1}u&Y_{1}Y_{1}u &\cdots&Y_{n}Y_{1}u\\
\cdots&\cdots&\cdots&\cdots&\cdots&\cdots\\
X_{1}Y_{1}u&\cdots&X_{n}Y_{1}u&Y_{1}Y_{1}u &\cdots&Y_{n}Y_{1}u
\end{array}\right).
\end{eqnarray*}
We also define
\begin{equation*}
\nabla_{H,s}^{2}u:=\frac{1}{2}\left[\nabla_{H}^{2}u+(\nabla_{H}^{2}u)^{T}\right]
\end{equation*} 
which is the symmetric part of the matrix $\nabla_{H}^{2}u$. Noticing that $\nabla_{H}^{2}u\in \mathcal{S}^{2n\times 2n}\oplus J\mathbb{R}$.

\subsection{}

In this subsection, we briefly recall a well-known regularization of semi-continuous functions in the CR setting which will be used later in the paper.

Assume $n\geq1$ and let $\Omega$ be an open bounded set in $\mathbb{H}^{n}$. For a function $v \in USC(\bar \Omega)$ and $\epsilon > 0$, we define the $\epsilon$-upper envelop of $v$ by
\begin{equation}
v^{\epsilon}(\xi)
	:=\max\limits_{\eta\in\bar\Omega}\Big\{v(\eta) -\frac{1}{\epsilon}|\eta^{-1}\circ\xi|_{H}^{4}\Big\}, \qquad \forall~\xi \in \bar \Omega.
		\label{Eq:UpEnvDef}
\end{equation}
Likewise, for a function $w \in LSC(\bar \Omega)$, its $\epsilon$-lower envelop is defined by
\begin{equation}
w_{\epsilon}(\xi):=\min\limits_{\eta\in\bar\Omega}\Big\{w(\eta) +\frac{1}{\epsilon}|\eta^{-1}\circ\xi|_{H}^{4}\Big\},\qquad\forall~\xi\in\bar\Omega.
	\label{Eq:LowEnvDef}
\end{equation}

We collect below some useful properties. The proof can be found in \cite{WangC}.

\begin{enumerate}[(i)]
\item\label{UpLowPropi} $v^\epsilon, w_\epsilon$ belong to $C(\bar\Omega)$, are monotonic in $\epsilon$ and 
\begin{equation}
\text{$v_\epsilon \rightarrow v$, $w_\epsilon \rightarrow w$ pointwise as $\epsilon \rightarrow 0$.}
	\label{Eq:UpLowConv}
\end{equation}
\item\label{UpLowPropii} $v^{\epsilon}$ and $w_{\epsilon}$ are punctually second order differentiable (see e.g. \cite{CabreCaffBook} for a definition) almost everywhere in $\Omega$ and 
\begin{equation}
\nabla^{2}v^{\epsilon}\geq-\frac{C}{\epsilon}I_{2n+1},\quad\nabla^{2}w_{\epsilon}\leq \frac{C}{\epsilon}I_{2n+1},\quad\mbox{ a.e. in }\Omega,\label{utr}
\end{equation}
where $C:=\|\nabla^{2}_{\xi}|\eta^{-1}\circ\xi|_{H}^{4}\|_{L^{\infty}(\overline{\Omega}\times\overline{\Omega})}$.

\item\label{UpLowPropiii} For any $\xi \in \Omega$, there exists $\xi^* = \xi^*(\xi) \in \bar \Omega$ such that 
$$v^\epsilon(\xi) = v(\xi^*)  - \frac{1}{\epsilon}|(\xi^{*})^{-1}\circ\xi|_{H}^{4} \text{ and } |(\xi^{*})^{-1}\circ\xi|_{H}^{4} \leq \epsilon(\max_{\bar\Omega} v - v(\xi)).$$
Likewise, for any $\xi \in \Omega$, there exists $\xi_* = \xi_*(x) \in \bar \Omega$ such that 
$$w_\epsilon(\xi) = w(\xi_*)  + \frac{1}{\epsilon}|(\xi_{*})^{-1}\circ\xi|_{H}^{4}\text{ and } |(\xi_{*})^{-1}\circ\xi|_{H}^{4} \leq \epsilon(w(\xi) - \min_{\bar\Omega} w).$$

\end{enumerate}

We conclude the section with a simple lemma about the stability of envelops with respect to semi-continuity.

\begin{lem}\label{lem:EquiSC}
Assume that $v \in USC(\bar\Omega)$ and $\inf_{\bar\Omega} v > -\infty$. Then for all sequences $\epsilon_j \rightarrow 0$ and $\xi_j \rightarrow \xi \in \Omega$, there holds
\[
\limsup_{j \rightarrow \infty} v^{\epsilon_j}(\xi_j) \leq v(\xi).
\]
Likewise, if $w \in LSC(\bar\Omega)$ and $\sup_{\bar\Omega} w <  +\infty$, then
\[
\liminf_{j \rightarrow \infty} w_{\epsilon_j}(\xi_j) \geq w(\xi).
\]
\end{lem}

\begin{proof} We will only show the first assertion. Assume by contradiction that there exist some sequences $\epsilon_j \rightarrow 0$, $\xi_j \rightarrow \xi \in \Omega$ such that
\[
v^{\epsilon_j}(\xi_j) \geq v(\xi) + 2\delta \text{ for some } \delta > 0.
\]
By the semi-continuity of $v$, there exists $\theta > 0$ such that 
\[
v(\eta) \leq v(\xi) + \delta \text{ for all } |\xi^{-1}\circ\eta|_{H} < \theta.
\]
By property \eqref{UpLowPropiii}, there exists $\hat \xi_j$ such that
\[
v^{\epsilon_j}(\xi_j) = v(\hat \xi_j) - \frac{1}{\epsilon_j}| \hat \xi_j^{-1}\circ \xi_j|_{H}^4 \text{ and } | \hat \xi_j^{-1}\circ \xi_j|_{H}^4 \leq \epsilon_j (\sup_{\bar\Omega} v - v(\xi_j)) \rightarrow 0,
\]
where we have used $\inf_{\bar\Omega} v > -\infty$. It then follows that $| \hat \xi_j^{-1}\circ \xi_j|_{H}< \theta$ for all sufficiently large $j$ and so
\[
v^{\epsilon_j}(\xi_j) \leq v(\hat \xi_j) \leq v(\xi) + \delta,
\]
which amounts to a contradiction.
\end{proof}

\section{The principle of propagation of touching points}\label{sec:CP}

In this section, we prove Theorem \ref{thm:CPQuad}. We will establish the propagation principle for more general operators of the form
\begin{equation}
F[\psi] = \nabla^2_{H,s}\psi + L(\cdot,\psi,\nabla_{H}\psi),
 \label{Eq:FDef}
\end{equation}
where $L: \Omega \times \RR \times \RR^{2n} \rightarrow \mathcal{S}^{2n \times 2n}$, under some structural assumptions on $L$ and $\Sigma$ which we will detail below. (Clearly, Definition \ref{Def:ViscositySolution} extends to this general setting.)

The following structural conditions on $(F,\Sigma)$ are directly motivated by the CR invariant operator $A[\psi]$. First, we assume that $\Sigma$ satisfies 
\begin{equation}
A\in \Sigma, c\in(0,1)\Rightarrow cA\in \Sigma.
\label{Eq:UCond*S}
\end{equation}
Second, we assume that, for every $R > 0$ and $\Lambda > 0$, there exist $m \geq 0$, $\bar \theta > 0$ and $C > 0$ such that, for $\xi \in \Omega$ and $p \in \RR^{2n}$,
\begin{align}
|\nabla_\xi L(\xi,s,p)| 
	&\leq C|p|^m,\quad |\nabla_p L(\xi,s,p)| \leq C|p|\qquad \forall~ |s| \leq R,
	\label{Eq:Lcond0}\\
0 \leq L(\xi,s', p) - L(\xi,s,p) 
	 &\leq C(s' - s)\,|p|^m\,I_{2n} \qquad \forall~ -R \leq s \leq s' \leq R,
	 \label{Eq:Lcond1}
\end{align}
\begin{align}
&p \cdot \nabla_p L(\xi, s,p) - L(\xi, s,p)
\nonumber\\
	&\qquad 
	+ \theta\Lambda |\nabla_{p} L(\xi, s,p)|\,I_{2n}  - \theta\,I_{2n}
	 \leq C p \otimes p - \frac{1}{C}\, |p|^m\,I_{2n} \qquad\forall~\theta \in [0,\bar\theta], |s| \leq R.
	 \label{Eq:Lcond2}
\end{align}
Note that, \eqref{Eq:Lcond1} and \eqref{Eq:Lcond2} should be understood as inequalities between real symmetric matrices: $M \leq N$ if and only if $N - M$ is non-negative definite. Also, \eqref{Eq:Lcond1} implies that $L$ is non-decreasing in $s$.

\begin{example}\label{Ex:QuadF} For all $\alpha, \beta, \gamma\in C^{0,1}_{loc}(\RR)$ such that $\beta(s) > \beta_0 > 0$ for some constant $\beta_0$, $\gamma\geq0$, $\alpha$ is non-decreasing and $\beta, \gamma$ is non-increasing, the operator
\[
F[\psi] = \nabla_{H,s}^2 \psi + \alpha(\psi)\,\nabla_{H}\psi \otimes \nabla_{H}\psi -\gamma(\psi)\,J\nabla_{H}\psi \otimes J\nabla_{H}\psi- \beta(\psi)\,|\nabla_{H}\psi|^2\,I_{2n}
\]
satisfies conditions \eqref{Eq:Lcond0}-\eqref{Eq:Lcond2}.
\end{example}

We now state our principle of propagation of touching points for operators of the form \eqref{Eq:FDef}.

\begin{thm}
Let $\Omega\subset\mathbb{H}^{n}$ ($n \geq 1$) be a non-empty bounded open set, $L: \Omega \times \RR \times \RR^{2n} \rightarrow \mathcal{S}^{2n\times 2n}$ be locally Lipschitz continuous and satisfy \eqref{Eq:Lcond0}, \eqref{Eq:Lcond1} and \eqref{Eq:Lcond2} for some $m > 1$, $F$ be given by \eqref{Eq:FDef} and $\Sigma$ be a non-empty open 
subset of $\mathcal{S}^{2n\times 2n}$ satisfying \eqref{Eq:UCondPos} and \eqref{Eq:UCond*S}. If $w \in LSC(\bar\Omega)$ and $v\in USC(\bar\Omega)$ are respectively a supersolution and a subsolution of \eqref{Eq:FpsiEq} in $\Omega$, and if $w \geq v$ in $\Omega$ and $w > v$ on $\partial\Omega$, then $w > v$ in $\Omega$.
\label{thm:CPNUE}
\end{thm}

Interchanging the role of $\psi$ and $-\psi$ and of $\Sigma$ and $\mathcal{S}^{2n \times 2n} \setminus (- \bar \Sigma)$ (where $-\bar \Sigma = \{-M: M \in \bar \Sigma\}$), we see that an analogous result holds if one replaces \eqref{Eq:UCond*S} by
\begin{equation}
A\in \Sigma, c\in(1,\infty) \Rightarrow cA\in \Sigma,
\label{Eq:UCond*SCat}
\end{equation}
and \eqref{Eq:Lcond2} by: for every $R > 0$ and $\Lambda > 0$, there exist positive constants $\bar \theta, C > 0$ such that, for $0 < \theta \leq \bar\theta$, $\xi \in \Omega$, $|s| \leq R$ and $p \in \RR^n$,
\begin{align}
&p \cdot \nabla_p L(\xi, s,p) - L(\xi, s, p ) 
\nonumber\\
	&\qquad\qquad - \theta\Lambda |\nabla_{p} L(\xi, s,p)|\,I_{2n}  + \theta\,I_{2n}
	 \geq -C p \otimes p + \frac{1}{C}\, |p|^m\,I_{2n}.
	\label{Eq:Lcond2Cat}
\end{align}
We then obtain an equivalent statement of Theorem \ref{thm:CPNUE}:

\begin{thm}
Let $\Omega\subset\mathbb{H}^{n}$ ($n \geq 1$) be a non-empty bounded open set, $L: \Omega \times \RR \times \RR^{2n} \rightarrow \mathcal{S}^{2n\times 2n}$ be locally Lipschitz continuous and satisfy \eqref{Eq:Lcond0}, \eqref{Eq:Lcond1} and \eqref{Eq:Lcond2Cat} for some $m > 1$, $F$ be given by \eqref{Eq:FDef} and $\Sigma$ be a non-empty open 
subset of $\mathcal{S}^{2n\times 2n}$ satisfying \eqref{Eq:UCondPos} and \eqref{Eq:UCond*SCat}. If $w \in LSC(\bar\Omega)$ and $v\in USC(\bar\Omega)$ are respectively a supersolution and a subsolution of \eqref{Eq:FpsiEq} in $\Omega$ and if $w \geq v$ in $\Omega$ and $w > v$ on $\partial\Omega$, then $w > v$ in $\Omega$.
\label{thm:CPNUECat}
\end{thm}

Assuming the correctness of the above theorem for the moment, we proceed with the 
\begin{proof}[Proof of Theorem \ref{thm:CPQuad}]
If $\beta > \beta_0 > 0$ and $\gamma\geq0$, the result is covered by Theorem \ref{thm:CPNUE}. If $\beta < -\beta_0 < 0$ and $\gamma\leq0$, the result is covered by Theorem \ref{thm:CPNUECat}. It remains to consider the case $\beta \equiv \gamma\equiv0$ and $\alpha$ is constant. The operator $F$ then takes the form
\[
F[\psi] = \nabla^2_{H,s}\psi + \alpha\,\nabla_{H}\psi \otimes \nabla_{H}\psi.
\]
When $\alpha \neq 0$, we note that the functions $\tilde w = \frac{\alpha}{|\alpha|} e^{\alpha w}$ and $\tilde v = \frac{\alpha}{|\alpha|} e^{\alpha v}$ satisfy $\tilde w \in LSC(\bar\Omega)$, $\tilde v\in USC(\bar\Omega)$ and, in view of \eqref{Eq:UCone},
\[
\nabla^2_{H,s}\tilde w = |\alpha|\,|\tilde w|\,F[w] \in \mathcal{S}^{2n\times 2n} \setminus \Sigma \text{ and } \nabla^2 _{H,s}\tilde v = |\alpha|\,|\tilde v| F[v] \in \bar \Sigma.
\]
Therefore, we can assume without loss of generality that $\alpha = 0$, i.e.
\[
F[\psi] = \nabla^2_{H,s} \psi.
\]

In this case, note that 
\begin{equation}
F[\psi + \mu\,|\xi|^2]  = F[\psi] + 2\mu(I_{2n}+4Jz\otimes Jz).
	\label{Eq:FVab=0}
\end{equation}
An easy adaption of the proof of Theorem \ref{thm:CPNUE} below (but using \eqref{Eq:FVab=0} instead of Lemma \ref{Lem:FVSub}) yields the result.
\end{proof}

We turn now to the proof of Theorem \ref{thm:CPNUE}.

%+++++++++++++++++%

\subsection{Error in regularizations}

The following result estimates the error to \eqref{Eq:FpsiEq} when making regularizations by lower and upper envelops.

\begin{prop}\label{key lemma}
Assume $n \geq 1$. Let $\Omega\subset\mathbb{H}^{n}$ be a bounded open set, $\Sigma$ be an open subset of $\mathcal{S}^{2n\times 2n}$ satisfying \eqref{Eq:UCondPos}, $L: \Omega \times \RR \times \RR^{2n} \rightarrow \mathcal{S}^{2n\times 2n}$ be a locally Lipschitz continuous function satisfying \eqref{Eq:Lcond0} and the second inequality in \eqref{Eq:Lcond1} for some $m \geq 0$, and $F$ be given by \eqref{Eq:FDef}. For any $M > 0$, there exists $a > 0$ such that if $w \in LSC(\Omega)$ is a supersolution of \eqref{Eq:FpsiEq} in $\Omega$ and if $w_\epsilon$ is punctually second order differentiable at a point $\xi \in \Omega$ and $|w_\epsilon(\xi)| + |w(\xi_*)| \leq M$, then
\begin{align*}
&F[w_\epsilon](\xi) - a( |\xi-\xi_*|+ \frac{1}{\epsilon}|(\xi_*)^{-1}\circ\xi|_{H}^{4})\,|\nabla_{H}w_\epsilon(\xi)|^m\,I_{2n} \in \mathcal{S}^{2n\times 2n} \setminus \Sigma.
\end{align*}
Analogously, if $v \in USC(\Omega)$ is a subsolution of \eqref{Eq:FpsiEq} in $\Omega$, and if $v^\epsilon$ is punctually second order differentiable at a point $\xi \in \Omega$ and $|v^\epsilon(\xi)| + |v(\xi^*)| \leq M$, then
\begin{align*}
&F[v^\epsilon](\xi) +  a( |\xi-\xi^*|+ \frac{1}{\epsilon}|(\xi^*)^{-1}\circ\xi|_{H}^{4})\,|\nabla_{H}v^\epsilon(\xi)|^m\,I_{2n}\in \Sigma.
\end{align*}
\end{prop}

\begin{proof}
We only give the proof of the first assertion. The second assertion can be proved in a similar way.

We have 
\begin{equation}
w_{\epsilon}(\xi\circ\eta)\geq w_{\epsilon}(\xi)+\nabla w_{\epsilon}(\xi)\cdot (\xi\circ\eta-\xi)+\frac{1}{2}(\xi\circ\eta-\xi)^{T}\nabla^{2}w_{\epsilon}(\xi)(\xi\circ\eta-\xi)+o(|\eta|^{2}),\quad\mbox{ as }|\eta|\rightarrow 0.\label{yre2}
\end{equation}
By the definition of $w_{\epsilon}$, we have
\begin{equation*}
w_{\epsilon}(\xi\circ\eta)\leq w(\xi_{*}\circ\eta) +\frac{1}{\epsilon}|(\xi_{*}\circ\eta)^{-1}\circ(\xi\circ\eta)|_{H}^{4}=w(\xi_{*}\circ\eta) +\frac{1}{\epsilon}|(\xi_{*})^{-1}\circ\xi|_{H}^{4},\label{yree}
\end{equation*}
and therefore, in view of (\ref{yre2}),
\begin{align*}
w(\xi_{*}\circ\eta) &\geq w_{\epsilon}(\xi\circ\eta) - \frac{1}{\epsilon}|(\xi_{*})^{-1}\circ\xi|_{H}^{4}\\
&\geq P_{\epsilon}(\eta)+o(|\eta|^{2}),\quad\mbox{ as }\eta\rightarrow 0,
\end{align*}
where $P_{\epsilon}$ is a quadratic polynomial with 
\begin{align*}
P_{\epsilon}(0)&=w_{\epsilon}(\xi) -\frac{1}{\epsilon}|(\xi_{*})^{-1}\circ\xi|_{H}^{4} = w(\xi_*), \nonumber \\
\nabla_{H}P_{\epsilon}(0)&=\nabla_{H}w_{\epsilon}(\xi),\\
\nabla^{2}_{H,s}P_{\epsilon}(0)&=\nabla^{2}_{H,s}w_{\epsilon}(\xi).
\end{align*}
Since $f(\eta):=w(\xi_{*}\circ\eta)$ is a viscosity supersolution of \eqref{Eq:FpsiEq} in $\tilde{\Omega}:=\{\eta\in\mathbb{H}^{n}:\xi_{*}\circ\eta\in\Omega\}$, we thus have 
\[
\nabla^2_{H,s}w_\epsilon(\xi) + L(\xi_*, w(\xi_*),\nabla_{H}w_{\epsilon}(\xi)) = F[P_{\epsilon}](0)\in \mathcal{S}^{2n\times 2n}\setminus \Sigma.\label{fanyang}
\]
On the other hand, in view of \eqref{Eq:Lcond0}, \eqref{Eq:Lcond1} and $w(\xi_*) = w_{\epsilon}(\xi) -\frac{1}{\epsilon}|(\xi_{*})^{-1}\circ\xi|_{H}^{4}  \leq w_\epsilon(\xi)$, 
\[
L(\xi_*, w(\xi_*),\nabla_{H}w_{\epsilon}(\xi))  - L(\xi, w_{\epsilon}(\xi),\nabla_{H}w_{\epsilon}(\xi)) \leq C(|\xi-\xi_*| + \frac{1}{\epsilon}|(\xi_{*})^{-1}\circ\xi|_{H}^{4})|\nabla_{H}w_\epsilon(\xi)|^m\,I_{2n}.
\]
The conclusion is readily seen thanks to \eqref{Eq:UCondPos}.
\end{proof}

%+++++++++++++++++%

\subsection{First variation of $F[\psi]$}

As mentioned in the introduction, we would like to perturb a given function $\psi$ to another function $\tilde\psi$ in such a way that $F[\tilde\psi] $ is bounded from above/below by a multiple of $F[\psi]$ and with a favorable excess term. This will be important in controlling error accrued in other parts of the proof of Theorem \ref{thm:CPNUE} (e.g. in regularizations).

\begin{lem}\label{Lem:FVSub}
Let $\Omega$ be an open bounded subset of $\HH^n$, $n \geq 1$, $L: \Omega \times \RR\times \RR^{2n} \rightarrow \mathcal{S}^{2n\times 2n}$ be a locally Lipschitz continuous function satisfying \eqref{Eq:Lcond1} and \eqref{Eq:Lcond2} for some $m > 1$, $F$ be given by \eqref{Eq:FDef}, and $\psi: \Omega \rightarrow \RR \cup \{\pm \infty\}$. For any $M > 0$, there exist positive constants $\mu_0, \alpha, \beta, \delta, K_0 > 0$, depending only on an upper bound of $M$, $L$ and $\Omega$, such that 
$$
\mu_0\,\beta\,\sup_\Omega e^{-\beta\psi} \leq \frac{1}{2},
$$
and, for any $0 < \mu < \mu_0$, $\tau \in \RR$, the function $\tilde\psi_{\mu,\tau} = \psi + \mu \,(e^{\alpha|z|^2} + e^{-\beta\psi} - \tau)$ satisfies
\[
F[\tilde\psi_{\mu,\tau}] \geq (1 - \mu\,\beta\,e^{-\beta\psi})F[\psi]
  	+ \mu\,K_0[(1 + |\nabla_{H}\psi|^m)\,I_{2n} + \nabla_{H} \psi \otimes \nabla_{H}\psi]
\]
in the set
\begin{align}
\Omega^{M,\delta} 
	&:= \Big\{\xi \in \Omega: \text{$\psi$ is punctually second order differentiable at $\xi$},\nonumber \\
		&\qquad\qquad |\psi(\xi)| \leq M, \text{ and }  e^{\alpha|z|^2} + e^{-\beta\psi(\xi)} - \tau \geq -\delta\Big\}.\label{Eq:OmegaMDef}
\end{align}
\end{lem}

\begin{proof} In the proof, $C$ will denote some large positive constant which may become larger as one moves from lines to lines but depends only on an upper bound for $M$, $L$ and $\Omega$. Eventually, we will choose large $\beta = \beta(C) > 0$, small $\alpha = \alpha(\beta, M, C) > 0$, and finally small $\mu_0 = \mu_0(\alpha, \beta,M, C) > 0$.

We set $\varphi(\xi) =\varphi(z,t)= e^{\alpha|z|^2}$, $f(\psi) = -  e^{-\beta\psi}$ and abbreviate  $\tilde\psi_\mu = \tilde\psi_{\mu,\tau} = \psi + \mu\,(\varphi - f(\psi) - \tau)$. Note that $f'(\psi) > 0$.

We assume in the sequel that $\alpha < 1$, $\delta < 1$ and
\begin{align}
\mu_0 \sup_\Omega [1 + f'(\psi)] &\leq \frac{1}{C} < \frac{1}{2}.
	\label{Eq:mu0Req1}
\end{align}

The following computation is done at a point in $\Omega^{M,\delta}$. We have
\begin{align*}
F[\tilde\psi_\mu]
	&\geq (1 - \mu\,f'(\psi)) F[\psi] - \mu\,f''(\psi) \nabla_{H}\psi \otimes \nabla_{H}\psi
		+ 2\mu \alpha\,\varphi\, I_{2n}
		\nonumber\\
		&\qquad\qquad + L(\xi,\tilde\psi_\mu,\nabla_{H}\tilde\psi_\mu) - (1 - \mu\,f'(\psi))  L(\xi,\psi,\nabla_{H}\psi).
\end{align*}
Noting that $\varphi - f(\psi) - \tau \geq -\delta$ in $\Omega^{M,\delta}$, we deduce from \eqref{Eq:Lcond1} that
\[
L(\xi,\tilde\psi_\mu,\nabla_{H}\tilde\psi_\mu) \geq L(\xi,\psi,\nabla_{H}\tilde\psi_\mu) - C\,\mu\,\delta\,(|\nabla_{H}\psi|^m + \mu^m\,\alpha^m\,\varphi^m)\,I_{2n}.
\]
Therefore,
\begin{align}
F[\tilde\psi_\mu]
	&\geq (1 - \mu\,f'(\psi)) F[\psi] - \mu\,f''(\psi) \nabla_{H}\psi \otimes \nabla_{H}\psi\nonumber\\
		&\qquad\qquad + 2\mu \alpha\,(1 - C\delta\mu^m\alpha^{m-1}\varphi^{m-1})\varphi\, I_{2n} - C\,\mu\,\delta|\nabla_{H}\psi|^m\,I_{2n}
		\nonumber\\
		&\qquad\qquad + L(\xi,\psi,\nabla_{H}\tilde\psi_\mu) - (1 - \mu\,f'(\psi))  L(\xi,\psi,\nabla_{H}\psi).
		\label{Eq:FtpsiEst1}
\end{align}

We proceed to estimate $L(\xi,\psi,\nabla_{H}\tilde\psi_\mu) - (1 - \mu\,f'(\psi))  L(\xi,\psi,\nabla_{H}\psi)$. For $0 \leq t \leq \mu$, let
\[
g(t) = \frac{L(\xi,\psi,\nabla_{H}\tilde\psi_t)}{1 - t f'(\psi)}. 
\]
We have
\begin{align*}
\frac{d}{dt} g(t)
	&\geq \frac{f'(\psi)}{(1 - t f'(\psi))^2} \Big[L(\xi,\psi,\nabla_{H}\tilde\psi_t) - \nabla_{H}\tilde\psi_t \cdot \nabla_p L(\xi,\psi,\nabla_{H}\tilde\psi_t)\\
		&\qquad\qquad - \frac{C\alpha \varphi}{f'(\psi)} |\nabla_p L(\xi,\psi,\nabla_{H}\tilde\psi_t)|\,I_{2n} \Big].
\end{align*}
Thus, in view of \eqref{Eq:Lcond2} and \eqref{Eq:mu0Req1}, if $\alpha, \beta$ and $\delta$ satisfy
\begin{align}
\alpha \sup_\Omega \varphi[\frac{1}{f'(\psi)} + 1]\,&\leq \frac{1}{C},\label{Eq:alphaReq}
\end{align}
then, with $R=M$, $\Lambda = 8C$ and $\theta = \frac{\alpha \varphi}{8f'(\psi)}$ in \eqref{Eq:Lcond2},
\begin{align*}
\frac{d}{dt} g(t)
	&\geq f'(\psi) \Big[-C \nabla_{H}\tilde\psi_t \otimes \nabla_{H}\tilde\psi_t + \frac{1}{C} |\nabla_{H}\tilde\psi_t|^m\,I_{2n}\Big] -  \frac{1}{2}\alpha\,\varphi \,I_{2n}\\
	&\geq f'(\psi) \Big[-C \nabla_{H}\psi \otimes \nabla_{H}\psi + \frac{1}{C} |\nabla_{H}\psi|^m\,I_{2n}\Big] -  \alpha\,\varphi \,I_{2n}.
\end{align*}
This implies
\begin{align}
&L(\xi,\psi,\nabla_{H} \tilde\psi_\mu) - (1 - \mu\,f'(\psi))  L(\xi,\psi,\nabla_{H}\psi)\nonumber\\
	&\qquad\qquad= (1 - \mu\,f'(\psi))[ g(\mu) - g(0)] \nonumber\\
	&\qquad\qquad\geq \mu\,f'(\psi) \Big[-C \nabla_{H}\psi \otimes \nabla_{H}\psi + \frac{1}{C} |\nabla_{H}\psi|^m\,I_{2n}\Big] - \mu\,\alpha\,\varphi\,I_{2n}.
	\label{Eq:FtpsiEst2}
\end{align}

Combining \eqref{Eq:FtpsiEst1} and \eqref{Eq:FtpsiEst2} and using \eqref{Eq:alphaReq}, we obtain
\begin{align}
F[\tilde\psi_\mu]
	&\geq (1 - \mu\,f'(\psi)) F[\psi]  +\frac{1}{2} \mu\, \alpha\,\varphi I_{2n}
		+  \frac{1}{C}\, \mu\,(f'(\psi) - C\delta)  |\nabla_{H}\psi|^m\,I_{2n}
		\nonumber\\
		&\qquad\qquad + \mu\,\big[-f''(\psi)  - Cf'(\psi)\big]\,\nabla_{H}\psi \otimes \nabla_{H}\psi 	.
		\label{Eq:FtpsiEst2Y}
\end{align}

We now fix $C$ and proceed with the choice of $\alpha, \beta, \delta$ and $\mu_0$. First, choosing $\beta \geq 2C$ and recalling the definition of $f$, we have
\[
-f''(\psi)  - Cf'(\psi) = \beta(\beta - C)e^{-\beta \psi} \geq \frac{1}{2}\beta\,f'(\psi).
\]
Next, choose $\alpha$ such that \eqref{Eq:alphaReq} is satisfied and choose $\delta$ such that $f'(\psi) - C\delta \geq \frac{1}{2}f'(\psi)$. Finally, choose $\mu_0$ such that \eqref{Eq:mu0Req1} holds. We hence obtain from \eqref{Eq:FtpsiEst2Y}  that
\[
F[\tilde\psi_{\mu}]
	\geq (1 - \mu\,f'(\psi)) F[\psi]
		+ \frac{1}{2} \mu  \,\alpha\,\varphi\, I_{2n} + \frac{1}{C}\, \mu\,f'(\psi)  |\nabla_{H}\psi|^m\,I_{2n} + \frac{1}{2}\beta\,\mu\,f'(\psi) \nabla_{H}\psi \otimes \nabla_{H}\psi.
\]
This completes the proof.
\end{proof}

\begin{lem}
Let $\Omega$ be an open bounded subset of $\HH^n$, $n \geq 1$, $L: \Omega \times \RR \times \RR^{2n} \rightarrow \mathcal{S}^{2n\times 2n}$ be a locally Lipschitz continuous function satisfying \eqref{Eq:Lcond1} and \eqref{Eq:Lcond2} for some $m > 1$, $F$ be given by \eqref{Eq:FDef}, and $\psi: \Omega \rightarrow \RR \cup \{\pm \infty\}$. For any $M > 0$, there exist positive constants $\mu_0, \alpha, \beta, \delta, K_0 > 0$, depending only on an upper bound of $M$, $L$ and $\Omega$, such that, for any $0 < \mu < \mu_0$, $\tau \in \RR$, the function $\hat\psi_{\mu,\tau} = \psi - \mu \,(e^{\alpha|z|^2} + e^{-\beta\psi} - \tau)$ satisfies
\[
F[\hat\psi_\mu] \leq (1 + \mu\,\beta\,e^{-\beta\psi})F[\psi]
  	- \mu\,K_0[(1 + |\nabla_{H}\psi|^m)\,I_{2n} + \nabla_{H}\psi \otimes \nabla_{H}\psi]
\]
in the set $\Omega^{M,\delta}$ defined by \eqref{Eq:OmegaMDef}.
\end{lem}

\begin{proof} The proof is similar to that of Lemma \ref{Lem:FVSub} and is omitted.
\end{proof}

%+++++++++++++++++%

\subsection{Proof of Theorem \ref{thm:CPNUE}}

Arguing by contradiction, we suppose that there exists $\gamma>0$ such that 
\begin{equation*}
\max\limits_{\bar\Omega}(v-w) = 0\quad  \text{ and }\quad (v-w)(\xi)\leq-\gamma,\quad\forall~\xi\in\overline{\Omega\setminus \Omega_{\gamma}}\label{lpl}
\end{equation*} 
where $\Omega_\gamma = \{\xi \in \Omega: \textrm{dist}_{H}(\xi, \partial\Omega) :=\inf\limits_{\eta\in\partial\Omega}|\eta^{-1}\circ\xi|_{H}> \gamma\}$. 

For $\epsilon > 0$, let $v^\epsilon$ and $w_\epsilon$ be the $\epsilon$-upper and $\epsilon$-lower envelops of $v$ and $w$ respectively (see Section \ref{Sec:Prelim}). We note that
\[
v \leq v^\epsilon \leq \max_{\bar\Omega} v < +\infty \text{ and } w \geq w_\epsilon \geq \min_{\bar\Omega} w > -\infty. 
\]

In the sequel, we use $C$ to denote some positive constant which depends on $\max_{\bar\Omega} v$, $\min_{\bar\Omega} w$, $L$ and $\Omega$ but is always independent of $\epsilon$.

By Lemma \ref{Lem:FVSub}, we can find $\bar \mu > 0$, $\delta > 0$ and a smooth positive function $f: \RR^{2n+1} \times \RR \rightarrow (0,\infty)$, depending only on $\max_{\bar\Omega} v$, $\min_{\bar\Omega} w$, $L$ and $\Omega$, such that $f$ is decreasing with respect to the $s$-variable, $\bar\mu \sup_\Omega |\partial_s f(\cdot, v^\epsilon)| \leq\frac{1}{2}$ and, for $\mu \in (0,\bar\mu)$, $\tau \in \RR$ and $\tilde v_{\epsilon,\tau} = v^\epsilon + \mu (f(\cdot,v^\epsilon) - \tau)$, there holds
\begin{equation}
F[\tilde v_{\epsilon,\tau}] \geq (1 - \mu|\partial_s f(\cdot, v^\epsilon)|)F[v^\epsilon]
  	+ \frac{\mu}{C}(1 + |\nabla_{H}v^\epsilon|^m)\,I_{2n}
			\label{Eq:Ftw}
\end{equation}
in the set
\begin{multline*}
\tilde\Omega_{\epsilon} := \Big\{\xi \in \Omega_{\gamma/2}: \text{$v^\epsilon$ is punctually second order differentiable at $\xi$},\\
	v^\epsilon(\xi) \geq \min_{\bar\Omega} w- 1 \text{ and } f(\xi,v^\epsilon(\xi)) - \tau \geq -\delta \Big\}.
\end{multline*}
Note that $\bar\mu$ and $\delta$ are independent of $\epsilon$. Furthermore, in view of \eqref{Eq:UpLowConv}, there exists $\bar \eta > 0$ independent of $\epsilon$ such that, for all small $\epsilon$ and $\eta \in (0,\bar\eta)$, one can (uniquely) find $\tau = \tau(\epsilon,\eta)$ such that the function $\zeta_{\epsilon,\eta} := \tilde v_{\epsilon,\tau} - w_\epsilon$ satisfies
\[
\max_{\bar\Omega} \zeta_{\epsilon,\eta} = \eta \text{ and } \zeta_{\epsilon,\eta} < -\frac{\gamma}{2} \text{ on }\overline{\Omega\setminus \Omega_{\gamma}}.
\]

Let $\Gamma_{\zeta_{\epsilon,\eta}^{+}}$ denote the concave envelope of $\zeta_{\epsilon,\eta}^{+}:=\max\{\zeta_{\epsilon,\eta},0\}$ on $\bar\Omega$. Then by \eqref{utr}, we have 
\begin{equation*}
\nabla^{2}\zeta_{\epsilon,\eta}\geq-\frac{C}{\epsilon}I_{2n+1}\quad\mbox{ a.e. in }\Omega_{\gamma}.
\end{equation*}
By \cite[Lemma 3.5]{CabreCaffBook}, we have 
\begin{equation*}
\int_{\{\zeta_{\epsilon,\eta}=\Gamma_{\zeta_{\epsilon,\eta}^{+}}\}}\mbox{det}(-\nabla^{2}\Gamma_{\zeta_{\epsilon,\eta}^{+}})>0,
\end{equation*}
which implies that the Lebesgue measure of $\{\zeta_{\epsilon,\eta}=\Gamma_{\zeta_{\epsilon,\eta}^{+}}\}$ is positive. Then there exists $\xi_{\epsilon,\eta}\in\{\zeta_{\epsilon,\eta}=\Gamma_{\zeta_{\epsilon,\eta}^{+}}\}\cap\Omega_{\gamma}$ such that both of $v^\epsilon$ and $w_\epsilon$ are punctually second order differentiable at $\xi_{\epsilon,\eta}$, 
\begin{equation}
0<\zeta_{\epsilon,\eta}(\xi_{\epsilon,\eta})\leq\eta,\label{Eq:29Dec16b}
\end{equation}
\begin{equation}
|\nabla\zeta_{\epsilon,\eta}(\xi_{\epsilon,\eta})| = |\nabla \tilde v_{\epsilon,\tau}(\xi_{\epsilon,\eta})- \nabla w_{\epsilon}(\xi_{\epsilon,\eta})| \leq C\eta,\label{Eq:29Dec16c}
\end{equation}
and
\begin{equation}
\nabla^{2}\zeta_{\epsilon,\eta}(\xi_{\epsilon,\eta})=\nabla^2 \tilde v_{\epsilon,\tau}(\xi_{\epsilon,\eta})- \nabla^2 w_{\epsilon}(\xi_{\epsilon,\eta})\leq 0.\label{Eq:29Dec162m}
\end{equation}

It follows from (\ref{Eq:29Dec16c}) and (\ref{Eq:29Dec162m}) that 
\begin{equation}
|\nabla_{H}\zeta_{\epsilon,\eta}(\xi_{\epsilon,\eta})| \leq C\eta,\label{Eq:29Dec16c'}
\end{equation}
and
\begin{equation}
\nabla^{2}_{H,s}\zeta_{\epsilon,\eta}(\xi_{\epsilon,\eta})=\nabla^2_{H,s}\tilde v_{\epsilon,\tau}(\xi_{\epsilon,\eta})- \nabla^2_{H,s}w_{\epsilon}(\xi_{\epsilon,\eta})\leq 0.\label{Eq:29Dec162m'}
\end{equation}

From \eqref{Eq:29Dec16b} and the definition of $\tilde v_{\epsilon,\tau}$, we have
\begin{equation}
	f(\xi_{\epsilon,\eta},v^\epsilon(\xi_{\epsilon,\eta})) - \tau 
	> \frac{1}{\mu}(w_\epsilon(\xi_{\epsilon,\eta}) - v^\epsilon(\xi_{\epsilon,\eta})).
	\label{Eq:f-tau->0}
\end{equation}
Note that, as $w \geq v$ in $\Omega$, Lemma \ref{lem:EquiSC} implies that
\[
\liminf_{\epsilon \rightarrow 0, \eta \rightarrow 0}[w_\epsilon(\xi_{\epsilon,\eta}) - v^\epsilon(\xi_{\epsilon,\eta})] \geq 0.
\]
Hence, by shrinking $\mu$ and $\bar\eta$ if necessary, we may assume for all small $\epsilon$ that
\[
f(\xi_{\epsilon,\eta},v^\epsilon(\xi_{\epsilon,\eta})) - \tau \geq -\delta, \qquad v^\epsilon(\xi_{\epsilon,\eta}) \geq \min_{\bar\Omega} w - 1, \quad \text{ and } w_\epsilon(\xi_{\epsilon,\eta}) \leq \max_{\bar\Omega} v + 1.
\]
We deduce that $\xi_{\epsilon,\eta} \in \tilde\Omega_{\epsilon,\delta}$ and thus obtain from \eqref{Eq:Ftw} that
\begin{equation}
F[\tilde v_{\epsilon,\tau}](\xi_{\epsilon,\eta}) \geq (1 - \mu|\partial_s f(\xi_{\epsilon,\eta}, v^\epsilon(\xi_{\epsilon,\eta}))|)F[v^\epsilon](\xi_{\epsilon,\eta})
  	+ \frac{\mu}{C}(1 + |\nabla_{H}v^\epsilon(\xi_{\epsilon,\eta})|^m)\,I_{2n}.
		\label{Eq:FtwInAction}
\end{equation}

Next, by (\ref{Eq:Lcond0}) and (\ref{Eq:Lcond1}), we have
\[
L(\xi_{\epsilon,\eta}, w_\epsilon(\xi_{\epsilon,\eta}),\nabla_{H}w_{\epsilon}(\xi_{\epsilon,\eta})) - L(\xi_{\epsilon,\eta}, \tilde v_{\epsilon,\tau}(\xi_{\epsilon,\eta}), \nabla_{H}\tilde v_{\epsilon}(\xi_{\epsilon,\eta}))
	\geq -C\eta(|\nabla_{H}v^{\epsilon}(\xi_{\epsilon,\eta})|^{m}+1)\,I_{2n}.
\]
This together with \eqref{Eq:29Dec162m} implies that
\begin{equation}
F[w_{\epsilon}](\xi_{\epsilon,\eta}) \geq F[\tilde v^{\epsilon,\tau}](\xi_{\epsilon,\eta}) - C\eta(|\nabla_{H}v^{\epsilon}(\xi_{\epsilon,\eta})|^{m}+1)\,I_{2n}.
	\label{Eq:CPFwtv>}
\end{equation}
Recalling \eqref{Eq:FtwInAction}, there holds
\begin{equation}
F[w_{\epsilon}](\xi_{\epsilon,\eta}) \geq (1 - \mu|\partial_s f(\xi_{\epsilon,\eta}, v^\epsilon(x_{\epsilon,\eta}))|)F[v^\epsilon](\xi_{\epsilon,\eta})
  	+ \frac{1}{C}(\mu-C\eta)(1 + |\nabla_{H}v^\epsilon(\xi_{\epsilon,\eta})|^m)\,I_{2n}.
		\label{Eq:Fweve>}
\end{equation}

We next claim that 
\begin{equation}
\liminf_{\epsilon \rightarrow 0, \eta \rightarrow 0} \frac{1}{\epsilon}\Big[|((\xi_{\epsilon,\eta})_*)^{-1} \circ \xi_{\epsilon,\eta}|_{H}^4 + |((\xi_{\epsilon,\eta})^*)^{-1} \circ\xi_{\epsilon,\eta}|^4_{H}\Big] \leq C\mu^2.
	\label{Eq:30Rep}
\end{equation}
Assuming this claim for now, we use Proposition \ref{key lemma} to find $a > 0$ independent of $\epsilon$ and $\eta$ such that one has, in $\Omega_\gamma$,
\begin{align}
F[w_\epsilon](\xi_{\epsilon,\eta}) - a( |\xi_{\epsilon,\eta}-(\xi_{\epsilon,\eta})_*|+ \frac{1}{\epsilon}|((\xi_{\epsilon,\eta})_*)^{-1}\circ\xi_{\epsilon,\eta}|_{H}^{4})\,|\nabla_{H}w_\epsilon(\xi_{\epsilon,\eta})|^m\,I_{2n}
	&\in \mathcal{S}^{2n\times 2n}\setminus \Sigma,
	\label{Eq:Fwe}\\
	F[v^\epsilon](\xi_{\epsilon,\eta}) +  a( |\xi_{\epsilon,\eta}-(\xi_{\epsilon,\eta})^*|+ \frac{1}{\epsilon}|((\xi_{\epsilon,\eta})^*)^{-1}\circ(\xi_{\epsilon,\eta})|_{H}^{4})\,|\nabla_{H}v^\epsilon(\xi_{\epsilon,\eta})|^m\,I_{2n}
	&\in \overline{\Sigma},
	\label{Eq:Fve}
\end{align}
where $(\xi_{\epsilon,\eta})_*$ and $(\xi_{\epsilon,\eta})^*$ are as in Section \ref{Sec:Prelim}. The relations \eqref{Eq:Fweve>}, \eqref{Eq:Fwe} and \eqref{Eq:Fve} amount to a contradiction for sufficiently small $\epsilon$, $\eta$ and $\mu$ thanks to \eqref{Eq:UCondPos} and \eqref{Eq:UCond*S}. Therefore, to conclude the proof it suffices to prove the claim \eqref{Eq:30Rep}.

Pick some $\eta(\epsilon)$ such that $\eta(\epsilon) \rightarrow 0$ as $\epsilon \rightarrow 0$. Pick a sequence $\epsilon_m \rightarrow 0$ such that, for $\xi_m := \xi_{\epsilon_m,\eta(\epsilon_m)}$, the sequence $\frac{1}{\epsilon_m}[((\xi_{m})^*)^{-1}\circ(\xi_{m})|_{H}^{4}+((\xi_{m})_*)^{-1}\circ(\xi_{m})|_{H}^{4}]$ converges to a limit which we will show to be no larger than $C\mu^2$. We will abbreviate $\tau_m := \tau(\epsilon_m, \eta(\epsilon_m))$, $v^m = v^{\epsilon_m}$, $w_m = w_{\epsilon_m}$. Without loss of generality, we may also assume that $\xi_m \rightarrow \xi_0 \in \Omega$, $f(\xi_m, v^m(\xi_m)) \rightarrow f_0$ and $\tau_m \rightarrow \tau_0$. 

As $\max_{\bar\Omega} \zeta_{\epsilon,\eta} = \eta$, we have in view of \eqref{Eq:UpLowConv} that
\begin{multline}
v(\xi_0) - w(\xi_0) + \mu(f(\xi_0, v(\xi_0)) -\tau_0)  \\
	 = \lim_{m \rightarrow \infty} \big\{v^{m}(\xi_0) - w_{m}(\xi_0) + \mu(f(\xi_0, v^{m}(\xi_0)) -\tau_m) \big\} \leq 0.
	 \label{Eq:vwmftLim}
\end{multline}
On the other hand, by \eqref{Eq:29Dec16b} and the fact that $f$ is decreasing in $s$, we have
\begin{align*}
f(\xi_0, \limsup_{m \rightarrow \infty} v^m(\xi_m))
	&\leq  f_0 
		= \lim_{m \rightarrow \infty} f(\xi_{m}, v^{m}(\xi_m))\\
	& \leq \limsup_{m \rightarrow \infty}  f(\xi_m, w_m(\xi_m) - \mu(f(\xi_m, v^m(\xi_m)) - \tau_m))\\
	& \leq   f(\xi_0, \liminf_{m \rightarrow \infty} w_m(\xi_m) - \mu(f_0 - \tau_0)),
\end{align*}
which implies, in view of Lemma \ref{lem:EquiSC} and the fact that $w \geq v$, that 
\[
f(\xi_0, w(\xi_0)) 
	\leq f(\xi_0, v(\xi_0)) 
	\leq f_0
	 \leq   f(\xi_0, w(\xi_0) - \mu(f_0 - \tau_0)),
\]
which further implies that
\[
0 \leq f_0 - f(\xi_0, v(\xi_0)) \leq C\mu.
\]
Together with \eqref{Eq:vwmftLim}, this implies that
\[
v(\xi_0) - w(\xi_0) + \mu(f_0 -\tau_0)  \leq C\mu^2.
\]

We are now ready to wrap up the argument. As $((\xi_\epsilon)^*)^{-1} \circ \xi_\epsilon \rightarrow 0$ and $((\xi_\epsilon)_*)^{-1} \circ \xi_\epsilon \rightarrow 0$ as $\epsilon \rightarrow 0$, we have $(\xi_m)_* \rightarrow \xi_0$ and $(\xi_m)^* \rightarrow \xi_0$. As $v$ is upper semi-continuous and $w$ is lower semi-continuous, we have
\[
\limsup_{m \rightarrow \infty} v((\xi_m)^*) \leq v(\xi_0) \text{ and } \liminf_{m \rightarrow \infty} w((\xi_m)_*) \geq w(\xi_0).
\]
Thus, by the left half of \eqref{Eq:29Dec16b},
\begin{align*}
0 
	&\leq \limsup_{m \rightarrow \infty} \frac{1}{\epsilon_m}[((\xi_{m})^*)^{-1}\circ(\xi_{m})|_{H}^{4}+((\xi_{m})_*)^{-1}\circ(\xi_{m})|_{H}^{4}]\\
	&\leq \limsup_{m \rightarrow \infty} \big\{v((\xi_{m})^*) - w((\xi_{m})_*) + \mu[f(\xi_{m},v^{\epsilon_m}(\xi_{m})) - \tau_{m}]\big\}\\
	&\leq v(\xi_0) - w(\xi_0) + \mu(f_0 - \tau_0) \leq C\mu^2.
\end{align*}
We have proved \eqref{Eq:30Rep}, and thus concluded the proof.
\hfill$\Box$

\section{Perron's method}\label{Sec:Perron}

We begin with the 
\begin{proof}[Proof of Theorem \ref{thm:Uniq}]
The conclusion is a direct consequence of Theorem \ref{thm:CPQuad} and Remark \ref{rem:SCP=>CP}.
\end{proof}

In the rest of this section, we prove Theorem \ref{thm:Perron}. We introduce some notations. For $O\subset \HH^n$,  $h: O\to [-\infty, +\infty]$,
let
$$
h^*(\xi):=\lim_{r\to 0^+}
\sup \{h(\eta)\ |\ \eta\in O, |\eta-\xi|<r\},
$$
$$
h_*(\xi):=\lim_{r\to 0^+}
\inf \{h(\eta)\ |\ \eta\in O, |\eta-\xi|<r\}.
$$
It is easy to see that, if $h^*(\xi) < +\infty$ for all $\xi \in O$, then $h^* \in USC(O)$. Likewise, if $h_*(\xi) > -\infty$ for all $\xi \in O$, then $h_* \in LSC(O)$.

$h^*$ is called the upper semicontinuous envelope of $h$,
it is the smallest upper semicontinuous
function satisfying $h\le h^*$ in $O$.
Similarly,
$h_*$ is called the lower semicontinuous envelope of $h$, it is the
largest
lower semicontinuous function satisfying $h\ge h_*$ in $O$.

Note that, for any constant $c$, $F[c] = 0 \in \partial \Sigma$. Thus, replacing $v$ by $\max(v,c)$ with some $c < \inf_{\partial\Omega} w$ and $w$ by $\min(w,c')$ with some $c' > \sup_{\partial\Omega} v$ if necessary, we can assume that
\[
-\infty < \inf_{\bar\Omega} v \leq \sup_{\bar\Omega} w < +\infty.
\]
Here we have used the fact that the maximum of two subsolutions is a subsolution and the minimum of two supersolutions is a supersolution.

Note that by hypotheses, $w\ge v$ in $\Omega$. Define
\begin{eqnarray}
u(\xi):=
&&
\inf\{h(\xi)\ |\
v\le h\le w\ \mbox{in}\ \overline \Omega,
h=v=w\ \mbox{on}\ \partial \Omega,\nonumber \\
&& \qquad \qquad h\in LSC(\overline \Omega),\
h\ \mbox{is a supersolution of \eqref{Eq:FpsiEq} in}\ \Omega\}.
\label{2new}
\end{eqnarray}

Clearly
$$
\inf_{ \overline \Omega} u
\ge \inf_{ \overline \Omega} v>-\infty.
$$

We will prove that the above defined $u$ satisfies the requirement of Theorem \ref{thm:Perron}.

\begin{lem}\label{lemC5-1new}
Let $O\subset \HH^n$ be an open set, $L: O \times \RR \times \RR^{2n} \rightarrow \mathcal{S}^{2n \times 2n}$ be continuous, $F$ be given by \eqref{Eq:FDef}, and
let ${\cal F}$ be a family of supersolutions of
\eqref{Eq:FpsiEq} in $O$.
Let
$$
g(\xi):=\inf\{h(\xi)\ |\
h\in {\cal F}\},
\ \ \ \xi\in O.
$$
Assume that $g_*(\xi)>-\infty\ \forall\ \xi\in O$.
Then $g_*$ is a supersolution of \eqref{Eq:FpsiEq} in $O$.
\end{lem}

\begin{proof}
Suppose for some $\xi\in O$ that there exists a polynomial $P$ of the form
$$
P(\eta):=a+p\cdot (\eta-\xi)+\frac 12 (\eta-\xi)^t M(\eta-\xi),
$$
with $a \in \RR$, $p\in \RR^{2n+1}$, $M\in {\cal S}^{(2n+1)\times (2n+1)}$, such that,
for some $\epsilon>0$,
\begin{equation}
P(\xi)=g_*(\xi) \text{ and } P(\eta)\le g_*(\eta)\ \ \forall\ |\eta-\xi|<\epsilon.
\label{C6-1new}
\end{equation}
We will show that
\begin{equation}
F[P](\xi)\in {\cal S}^{2n\times 2n} \setminus  \Sigma.
\label{C6-2new}
\end{equation}
It is standard that this implies that $g_*$ is a supersolution of \eqref{Eq:FpsiEq} in the sense of Definition \ref{Def:ViscositySolution}.

By the definition of $g_*$, there exists
$r_i\to 0^+$, $|\xi^{(i)}-\xi|<r_i$ such that
$$
\inf_{B_{r_i}(\xi)} g \leq g(\xi^{(i)}) \leq \inf_{B_{r_i}(\xi)} g + \frac{1}{i} \le g_*(\xi)  + \frac{1}{i} \text{ and }
 g(\xi^{(i)})\to g_*(\xi).
$$
Moreover,
 there exists
$h_i\in {\cal F}$, such that
$h_i\ge g \ge g_*$ and
$$
0\le h_i(\xi^{(i)})-g(\xi^{(i)})<\frac 1i.
$$
We see from the above that
$$
h_i\ge g\ge g_*\ge
P \quad \mbox{in}\ B_\epsilon(\xi),
$$
and
$$
h_i(\xi^{(i)})\to g_*(\xi)=P(\xi).
$$

For every $0<2\delta<\min\{\epsilon,
dist(x, \partial O)\}$,
consider
$$
P_\delta(\eta):= P(\eta)-\delta|\eta-\xi|^2.
$$
Then
$$
h_i\ge P_\delta\ \  \mbox{in}\ B_\epsilon(\xi),
\quad 
h_i\ge P_\delta+ \delta^3\ \  \mbox{in}\ B_\epsilon(\xi)\setminus
B_\delta(\xi), \text{ and }
h_i(\xi^{(i)})-P_\delta(\xi^{(i)})
\to 0.
$$

It follows that
 there exists $\beta_i=\circ(1) \ge 0$ and $\xi^{(i)*}\in B_\delta(\xi)$ such that
\begin{equation}\label{C7-0new}
h_i(\eta)\ge P_\delta(\eta) + \beta_i,\ \ \mbox{in}\ B_\epsilon(\xi),
 \qquad
h_i(\xi^{(i)*})=P_\delta(\xi^{(i)*}) + \beta_i.
\end{equation}
As $h_i$ is also a supersolution of \eqref{Eq:FpsiEq} in $O$.  Thus,
\begin{equation}
F[P_\delta + \beta_i](\xi^{(i)*})
\in {\cal S}^{2n\times 2n}\setminus 
 \Sigma.
\label{C7-1new}
\end{equation}

\noindent{\bf Claim.}\  $\xi^{(i)*}\to \xi$.

\medskip

Indeed, after passing to a subsequence,
$\xi^{(i)*}\to \bar \xi$,  for some $\bar \xi$ satisfying
$|\bar \xi-\xi|\le \delta.$
By \eqref{C7-0new} and the definition of $g$ and $g_*$,
$$
g_*(\xi^{(i)*}) - \beta_i \le h_i(\xi^{(i)*}) - \beta_i=P_\delta(\xi^{(i)*}).
$$
Sending $i$ to infinity in the above,
and using the lower-semicontinuity property of $g_{*}$,
we have
$
g_*(\bar \xi) \le P_\delta(\bar \xi)=P(\bar \xi)-\delta|\bar \xi-\xi|^2.
$
On the other hand, $P(\bar \xi)\le g_*(\bar \xi)$ according to
\eqref{C6-1new}.
Thus $\bar \xi=\xi$, and the claim is proved.

\medskip

With the convergence of $\xi^{(i)*}$ to $\xi$ and of $\beta_i$ to $0$, sending $\delta$ to $0$ and $i$ to $\infty$ in
\eqref{C7-1new} yields
\eqref{C6-2new}.
Lemma \ref{lemC5-1new} is established.
\end{proof}

\bigskip

\begin{proof}[Proof of Theorem \ref{thm:Perron}]
We know that
\begin{equation}
\max(v,u_*)\le u\le \min(u^*, w),
\qquad\mbox{in}\ \overline \Omega,
\label{C9-1new}
\end{equation}
where $u$ is defined by \eqref{2new}.  Clearly,
\begin{equation}
v= u_*= u=
u^*= w,
\qquad\mbox{on}\ \partial \Omega,
\label{C9-2new}
\end{equation}
By Lemma \ref{lemC5-1new},
$u_*$ is a supersolution of \eqref{Eq:FpsiEq} in $\Omega$. By the comparison principle Theorem \ref{thm:CPQuad}~(b), $u_* \geq v$. Hence, by the definition of $u$, $u\le u_*$ 
in $\overline \Omega$.
Thus $u=u_*$ in $\overline \Omega$, and $u$ is a supersolution of \eqref{Eq:FpsiEq} 
in $\Omega$.

Note that
\[
\sup_{\bar\Omega} u^* \leq \sup_{\bar\Omega} w < +\infty.
\]
\medskip

\noindent{\bf Claim.}\
$u^*$ is a subsolution of \eqref{Eq:FpsiEq}  in $\Omega$.

\medskip

To prove this claim, we follow Ishii's argument (\cite{Ishii89-CPAM}). Indeed, if the claim does not hold, there exist $\xi\in \Omega$ and some quadratic polynomial
$$
P(\eta)=a+p\cdot (\eta-\xi)+\frac 12 (\eta-\xi)^t M (\eta-\xi),
$$
with $a \in \RR$, $p\in \RR^{2n+1}$, $M\in {\cal S}^{(2n+1)\times (2n+1)}$, such that
for some $\bar \epsilon>0$ 
\begin{equation}
P(\eta)\ge u^*(\eta)\ \ \mbox{for}\ \eta\in
 B_{\bar\epsilon}(\xi),\qquad
P(\xi)=u^*(\xi),
\label{C10-1new}
\end{equation}
but
\begin{equation}
F[P](\xi)\in {\cal S}^{ 2n\times 2n}\setminus \overline \Sigma.
\label{C10-2new}
\end{equation}

Since ${\cal S}^{ 2n\times 2n}\setminus \overline \Sigma$ 
is open, there exists $0<2\bar\delta<
\min\{\bar \epsilon^2,  1\}$ such that
for all $0<\delta<\bar\delta$, the function
$$
P_\delta(\eta):=P(\eta)+\delta|\eta-\xi|^2 -\delta^2
$$
satisfies
\begin{equation}
P_\delta(\xi)=P(\xi)-\delta^2<u^*(\xi),
\label{C11-0new}
\end{equation}
and
\begin{equation}
F[P_\delta](\eta)\in {\cal S}^{ 2n\times 2n}\setminus \overline  \Sigma,\qquad
\forall\ |\eta-\xi|<\delta^{1/9}.
\label{C10-3new}
\end{equation}
Clearly,
\begin{equation}
P_\delta(\eta)>P(\eta),\qquad
\forall \ |\eta-\xi|\ge  \delta^{1/5}.
\label{C11-1new}
\end{equation}

Define
$$
\hat u(\eta):=
\left\{
\begin{array}{lr}
\displaystyle{
\min\{u(\eta), P_\delta(\eta)\},
}&
\mbox{if}\ |\eta-\xi|<\delta^{1/5},\\
u(\eta), &
\mbox{if}\ |\eta-\xi|\ge \delta^{1/5}.
\end{array}
\right.
$$
By \eqref{C10-3new},
$P_\delta$ is a supersolution of \eqref{Eq:FpsiEq}  in $\{\eta: |\eta-\xi|<\delta^{1/9}\}$.
By \eqref{C11-1new}, and using $P\ge u^*\ge u$, we have
$$
\hat u(\eta)=u(\eta)=
\min\{u(\eta), P_\delta(\eta)\},\qquad
\delta^{1/5}\le
|\eta-\xi|\le \delta^{1/6}.
$$
It follows that $\hat u$, being the minimum
 of two supersolutions,
 is a supersolution of \eqref{Eq:FpsiEq}  in $\Omega$ ( see proposition A.2 in \cite{WangBo2016}), and,
because of the definition of $u$,
\begin{equation}
u\le \hat u\qquad\mbox{in}\ \Omega.
\label{C12-1new}
\end{equation}
On the other hand we see from
  \eqref{C11-0new},   the definition of $\hat u$ and \eqref{C12-1new}
that there exists $\epsilon\in (0, \delta^{1/5})$ such that
$$
u(\eta)\le
\hat u(\eta)\le P_\delta(\eta)<u^*(\xi)-\epsilon,\qquad
\forall\ |\eta-\xi|<\epsilon.
$$
Thus
$$
u^*(\xi)
=\lim_{r\to 0^+}
\sup\{u(\eta)\ |\ |\eta-\xi|<r\}
\le u^*(\xi)-\epsilon,
$$
a contradition.
The claim is proved, i.e. $u^*$ is a subsolution of \eqref{Eq:FpsiEq}  in $\Omega$. 

\bigskip

Now we have proved that $u_* = u$ and $u^*$ are respectively
supersolution and subsolution of \eqref{Eq:FpsiEq}  in $\Omega$, and $u_*=u^*$
on $\partial \Omega$. By the comparison principle Theorem \ref{thm:CPQuad}~(b), $u^* \leq u$ in $\Omega$ and so $u = u_* = u^*$ is a solution of \eqref{Eq:FpsiEq}.
\end{proof}

\noindent{\bf{\large Acknowledgments.}} Li is partially supported by NSF grant DMS-1501004. Wang is partially supported by NNSF (11701027) and Beijing Institute of Technology Research Fund Program for Young Scholars. 

%\bibliography{paris}{}

\begin{thebibliography}{10}

\bibitem{AmendolaGaliseVitolo13-DIE}
{\sc M.~E. Amendola, G.~Galise, and A.~Vitolo}, {\em Riesz capacity, maximum
  principle, and removable sets of fully nonlinear second-order elliptic
  operators}, Differential Integral Equations, 26 (2013), pp.~845--866.
  
  \bibitem{BCM12}
{\sc E.~Barbosa,  L. G. Carneiro and M.~Montenegro},   {\em The k-Yamabe problem on CR manifolds},({\noopsort{a}}2012).
\newblock http://arxiv.org/abs/1205.1840v1.  
  
    
\bibitem{BardiDaLio98}
{\sc M.~Bardi and F.~Da~Lio}, {\em Propagation of maxima and strong maximum principle for degenerate elliptic equations}, Proc. of the eighth Tokyo Conference on Nonlinear PDE, (1998).  


\bibitem{BardiDaLio99-AM}
{\sc M.~Bardi and F.~Da~Lio}, {\em On the strong maximum principle for fully
  nonlinear degenerate elliptic equations}, Arch. Math. (Basel), 73 (1999),
  pp.~276--285.

\bibitem{HBDolcettaPorretaRosi15-JMPA}
{\sc H.~Berestycki, I.~Capuzzo~Dolcetta, A.~Porretta, and L.~Rossi}, {\em
  Maximum principle and generalized principal eigenvalue for degenerate
  elliptic operators}, J. Math. Pures Appl. (9), 103 (2015), pp.~1276--1293.

\bibitem{BirindelliDemengel04-AFSTM}
{\sc I.~Birindelli and F.~Demengel}, {\em Comparison principle and {L}iouville
  type results for singular fully nonlinear operators}, Ann. Fac. Sci. Toulouse
  Math. (6), 13 (2004), pp.~261--287.

\bibitem{BirindelliDemengel07-CPAA}
\leavevmode\vrule height 2pt depth -1.6pt width 23pt, {\em Eigenvalue, maximum
  principle and regularity for fully non linear homogeneous operators}, Commun.
  Pure Appl. Anal., 6 (2007), pp.~335--366.

\bibitem{CabreCaffBook}
{\sc L.~Caffarelli and X.~Cabr{\'e}}, {\em Fully nonlinear elliptic equations},
  vol.~43 of American Mathematical Society Colloquium Publications, American
  Mathematical Society, Providence, RI, 1995.

\bibitem{CLN2009}
{\sc L. Caffarelli, Y. Y. Li and L. Nirenberg}, {\em Some remarks on singular solutions of nonlinear elliptic equations. I}, Journal of Fixed Point Theory and Applications 5 (2009), pp.~353--395.

\bibitem{CLN2013}
\leavevmode\vrule height 2pt depth -1.6pt width 23pt, {\em Some remarks on singular solutions of nonlinear elliptic equations. III: viscosity solutions, including parabolic operators}, Comm. Pure Appl. Math. 66 (2013), pp.~109--143.

\bibitem{CNS1985}
{\sc L. Caffarelli, Y. Y. Li and L. Nirenberg}, {\em The Dirichlet problem for nonlinear second-order elliptic equations. III. Functions of the eigenvalues of the Hessian}, Acta Math. 155 (1985), no. 3-4, pp.~261--301. 

\bibitem{CNS1986}
\leavevmode\vrule height 2pt depth -1.6pt width 23pt, {\em The Dirichlet problem for the degenerate Monge-Ampère equation}, Rev. Mat. Iberoamericana 2 (1986), no. 1-2, pp.~19--27.

\bibitem{UserGuide}
{\sc M.~G. Crandall, H.~Ishii, and P.-L. Lions}, {\em User's guide to viscosity
  solutions of second order partial differential equations}, Bull. Amer. Math.
  Soc. (N.S.), 27 (1992), pp.~1--67.
  

  
\bibitem{DolcettaVitolo07-MC}
{\sc I.~C. Dolcetta and A.~Vitolo}, {\em On the maximum principle for viscosity
  solutions of fully nonlinear elliptic equations in general domains},
  Matematiche (Catania), 62 ({\noopsort{a}}2007), pp.~69--91.

\bibitem{DolcettaVitolo-preprint}
\leavevmode\vrule height 2pt depth -1.6pt width 23pt, {\em The weak maximum
  principle for degenerate elliptic operators in unbounded domains},
  ({\noopsort{b}}preprint).
  
\bibitem{GLN2018}
{\sc Maria del Mar Gonzalez, Y. Y. Li and L.~Nguyen}, {\em Existence and uniqueness to a fully non-linear version of the Loewner-Nirenberg problem} ({\noopsort{a}}2018).
\newblock http://arxiv.org/abs/1804.08851v1.  


\bibitem{HartmanNirenberg}
{\sc P.~Hartman and L.~Nirenberg}, {\em On spherical image maps whose
  {J}acobians do not change sign}, Amer. J. Math., 81 (1959), pp.~901--920.

\bibitem{HarveyLawsonSurvey2013}
{\sc F.~R. Harvey and H.~B. Lawson, Jr.}, {\em Existence, uniqueness and
  removable singularities for nonlinear partial differential equations in
  geometry}, in Surveys in differential geometry. {G}eometry and topology,
  vol.~18 of Surv. Differ. Geom., Int. Press, Somerville, MA, 2013,
  pp.~103--156.
  
  \bibitem{HarveyLawson2016}
  \leavevmode\vrule height 2pt depth -1.6pt width 23pt,  {\em Characterizing the strong maximum principle for constant coefficient subequations}, Rend. Mat. Appl., (7) 37 (2016), no. 1?2, pp.~63--104.
  
  
  
  
\bibitem{Ishii89-CPAM}
{\sc H.~Ishii}, {\em On uniqueness and existence of viscosity solutions of
  fully nonlinear second-order elliptic {PDE}s}, Comm. Pure Appl. Math., 42
  (1989), pp.~15--45.

\bibitem{IshiiLions90-JDE}
{\sc H.~Ishii and P.-L. Lions}, {\em Viscosity solutions of fully nonlinear
  second-order elliptic partial differential equations}, J. Differential
  Equations, 83 (1990), pp.~26--78.

\bibitem{Jensen88-ARMA}
{\sc R.~Jensen}, {\em The maximum principle for viscosity solutions of fully
  nonlinear second order partial differential equations}, Arch. Rational Mech.
  Anal., 101 (1988), pp.~1--27.



\bibitem{KawohlKutev98-AM}
{\sc B.~Kawohl and N.~Kutev}, {\em Strong maximum principle for semicontinuous
  viscosity solutions of nonlinear partial differential equations}, Arch. Math.
  (Basel), 70 (1998), pp.~470--478.

\bibitem{KawohlKutev00-FE}
\leavevmode\vrule height 2pt depth -1.6pt width 23pt, {\em Comparison principle
  and {L}ipschitz regularity for viscosity solutions of some classes of
  nonlinear partial differential equations}, Funkcial. Ekvac., 43 (2000),
  pp.~241--253.

\bibitem{KawohlKutev07-CPDE}
\leavevmode\vrule height 2pt depth -1.6pt width 23pt, {\em Comparison principle
  for viscosity solutions of fully nonlinear, degenerate elliptic equations},
  Comm. Partial Differential Equations, 32 (2007), pp.~1209--1224.

\bibitem{KoikeKosugi15-CPAA}
{\sc S.~Koike and T.~Kosugi}, {\em Remarks on the comparison principle for
  quasilinear {PDE} with no zeroth order terms}, Commun. Pure Appl. Anal., 14
  (2015), pp.~133--142.

\bibitem{KoikeLey11-JMAA}
{\sc S.~Koike and O.~Ley}, {\em Comparison principle for unbounded viscosity
  solutions of degenerate elliptic {PDE}s with gradient superlinear terms}, J.
  Math. Anal. Appl., 381 (2011), pp.~110--120.
  
\bibitem{LiLi2003}
{\sc A. Li and Y.~Y. Li}, {\em On some conformally invariant fully nonlinear equations}, Comm. Pure Appl. Math., 56 (2003), pp.~1416--1464.
  
  
\bibitem{Li07-ARMA}
{\sc Y.~Y. Li}, {\em Degenerate conformally invariant fully nonlinear elliptic equations}, Arch. Rational. Mech. and Anal. 186 (2007), pp.~25--51.

\bibitem{Li09-CPAM}
\leavevmode\vrule height 2pt depth -1.6pt width 23pt, {\em Local gradient estimates of solutions to some conformally
  invariant fully nonlinear equations}, Comm. Pure Appl. Math., 62 (2009),
  pp.~1293--1326.
\newblock (C. R. Math. Acad. Sci. Paris 343 (2006), no. 4, 249--252).

\bibitem{LiMonticelli-JDE}
{\sc Y.~Y. Li and D.~D.~Monticelli}, {\em On fully nonlinear CR invariant equations on the heisenberg group}, Journal of Differential Equations, 252 (2012),
  pp.~1309--1349.




\bibitem{LiNgWang-arxiv}
{\sc Y.~Y. Li, L.~Nguyen, and B. Wang}, {\em Comparison principles and Lipschitz regularity for some nonlinear degenerate elliptic equations},  Calculus of Variations and Partial Differential Equations, 57 (2018), no. 4, 57:96.

\bibitem{LN}
{\sc Y.~Y. Li and L. Nirenberg}, {\em A miscellany}, in Percorsi incrociati (in ricordo di Vittorio Cafagna), Collana Scientifica di Ateneo, Universita di Salerno, 2010, pp.~193--208. 
\newblock http://arxiv.org/abs/0910.0323.

\bibitem{LW}
{\sc Y.~Y. Li and B. Wang}, {\em Strong comparison principles for some nonlinear degenerate elliptic equations}, Acta Math. Sci. Ser. B (Engl. Ed. ) 38 (2018), no. 5, pp. 1583--1590.

\bibitem{MM}
{\sc  P. Mastrolia and D. D. Monticelli}, {\em On the relation between conformally invariant operators and some geometric tensors}, Rev. Mat. Iberoam. 31 (2015), no. 1, pp.~303--312.



\bibitem{Trudinger88-RMI}
{\sc N.~S. Trudinger}, {\em Comparison principles and pointwise estimates for
  viscosity solutions of nonlinear elliptic equations}, Rev. Mat.
  Iberoamericana, 4 (1988), pp.~453--468.

\bibitem{V}
{\sc J. Viaclovsky},  {\em Conformal geometry, contact geometry, and the calculus of variations}, Duke Math. J. 101 (2000), pp.~283--316.

  
  \bibitem{WangBo2016}
{\sc B. Wang}, {\em Removable singularities for degenerate elliptic Pucci operator on the Heisenberg group}, Nonlinear Analysis: Theory, Methods \& Applications, 160 (2017), pp.~177--190.   

  \bibitem{WangC}
  {\sc C. Wang}, {\em The comparsion principle for viscosity solutions of fully nonlinear subelliptic equations in Carnot groups}, (2003).
  \newblock http://arxiv:math/0309078v1.
  
    


\end{thebibliography}
%\bibliographystyle{siam}

\newcommand{\noopsort}[1]{}

\end{document}